\documentclass[a4paper, 11pt, reqno, final]{amsart}
\usepackage{amsmath}
\usepackage{amssymb}
\usepackage{upgreek}
\usepackage{comment}
\usepackage{tikz}
\usepackage{color}
\usepackage[T1]{fontenc}
\usepackage[textwidth=440pt, textheight=650pt]{geometry}
\usepackage{enumerate}
\usepackage[hypcap=false]{caption}
\usepackage{graphicx}
\usepackage{showlabels}
\usepackage{hyperref}

\DeclareMathOperator{\argmax}{argmax}
\DeclareMathOperator{\argmin}{argmin}
\DeclareMathOperator{\N}{\mathbb{N}}
\DeclareMathOperator{\bbS}{\mathbb{S}}
\DeclareMathOperator{\bbT}{\mathbb{T}}
\DeclareMathOperator{\E}{\mathbb{E}}
\DeclareMathOperator{\bbK}{\mathbb{K}}

\DeclareMathOperator{\R}{\mathbb{R}}
\DeclareMathOperator{\Prob}{\mathbb{P}}

\DeclareMathOperator{\M}{\mathbf{m}_{\infty}}

\DeclareMathOperator{\dimA}{\dim\sb{\mathrm{A}}}

\DeclareMathOperator{\ex}{\mathbf{e}}

\DeclareMathOperator{\eps}{\varepsilon}
\DeclareMathOperator{\from}{\colon\;}

\DeclareMathOperator{\diam}{diam}

\DeclareMathOperator{\Cov}{Cov}

\newtheorem{theorem}{Theorem}[section]
\newtheorem{corollary}[theorem]{Corollary}
\newtheorem{lemma}[theorem]{Lemma}
\newtheorem{proposition}[theorem]{Proposition}
\newtheorem{question}[theorem]{Question}
\newtheorem{conjecture}[theorem]{Conjecture}
\newtheorem{definition}[theorem]{Definition}

\numberwithin{equation}{section}

\title[Quasi-symmetric embeddings of the Brownian map and CRTs]{On quasisymmetric embeddings of the Brownian map and continuum trees}
\author[S.~Troscheit]{Sascha Troscheit}
\thanks{ST was initially supported by \emph{NSERC Grants} 2016-03719 and 2014-03154, and the
University of Waterloo.}
\address[Sascha Troscheit]{Faculty of Mathematics\\University of Vienna\\Oskar Morgenstern Platz~1\\1090 Wien\\Austria.}
\email{sascha.troscheit@univie.ac.at}
\urladdr{https://www.mat.univie.ac.at/~troscheit/}
\date{\today}

\begin{document}
\begin{abstract}
  The Brownian map is a model of random geometry on the sphere and as such
  an important object in probability theory and physics.
  It has been linked to Liouville Quantum Gravity and much research has been devoted to it. One open question asks for a
  canonical embedding of the Brownian map into the sphere or other, more abstract, metric spaces.
  Similarly, Liouville Quantum Gravity has been shown to be ``equivalent'' to the Brownian map but
  the exact nature of the correspondence (i.e.\ embedding) is still unknown.
  In this article we show that any embedding of the Brownian map or continuum random tree
  into $\R^d$, $\bbS^d$, $\bbT^d$, or more generally any
  doubling metric space, cannot be quasisymmetric. We achieve this with the aid of dimension theory
  by identifying a metric structure that is invariant under
  quasisymmetric mappings (such as isometries) and which implies infinite Assouad dimension.

  We show, using elementary methods, that this structure is almost surely present in the Brownian continuum random tree and the
  Brownian map.
  We further show that snowflaking the metric is not sufficient to find an embedding and 
  discuss continuum trees as a tool to studying ``fractal functions''.
\end{abstract}

\maketitle

\section{Introduction}
\label{sect:introduction}
Over the past few years two important models of random geometry of the sphere $\bbS^2$ emerged. 
One was originally motivated by string theory and conformal field theory in physics
and is known as Liouville Quantum Gravity.
The other is known as the Brownian map, which originated from the study of large random planar
maps. 
The Brownian map turned out to be a universal limit of random planar maps of $\bbS^2$ and both
models have attracted a
great deal of interest over the past few years.

We say that a {map} is planar if it is an embedding of a finite connected graph onto $\bbS^2$ with no edge-crossings.
A {quadrangulation} is a planar map where
all its faces are incident to exactly four edges, where edges incident to only one face are counted
twice (for both ``sides'' of the edge).
Let $Q_n$ be the set of all rooted\footnote{A graph is said to be \textbf{rooted} if there exists a
unique oriented and distinguished edge, called the root.} quadrangulations with $n$ vertices. This set is finite and we randomly choose a specific quadrangulation $q_n\in Q_n$ with the uniform
probability measure. Let $d_{gr}$ be the graph metric\footnote{The distance $d_{gr}(v,w)$ between two
  vertices $v,w$ in a finite connected graph is equal to the length of the shortest path between $v$
and $w$. If $v=w$, we set $d_{gr}=0$.}. It was independently proven by Le Gall~\cite{LeGall13}
and by Miermont~\cite{Miermont13} that the sequence 
$(q_n,  n^{-1/4}d_{gr})$ converges in distribution to a random
limit object $(\M,D)$, known as the Brownian map, in the space isometry classes of all compact sets $(\bbK, d_{GH})$ equipped with
the Gromov-Hausdorff distance. Remarkably, Le Gall also established that
this object is universal in the sense that not
just quadrangulations, but also uniform triangulations and $2n$-angulations ($n\geq 2$) converge in distribution to
the same object, up to a rescaling that only depends on the type of $p$-angulation, 
see \cite{LeGall13} and the survey \cite{LeGall14}. Notably, the Brownian map is homeomorphic to
$\bbS^2$ but has Hausdorff dimension $4$, indicating that the homeomorphism is highly singular.
Finding such a canonical mapping is still an active research area and in this article we show, using
elementary facts from probability and dimension theory, that the Brownian map has an almost sure
metric property that we coin \emph{starry}. This property implies that the Brownian map and its
images under quasisymmetric mappings has infinite Assouad dimension and hence cannot be embedded by
quasisymmetric mappings (such as isometries) into finite dimensional manifolds.

The other important model of random surfaces is known as Liouville Quantum Gravity (LQG) which is defined in
terms of a real parameter $\gamma$. The LQG is a random geometry based on the Gaussian Free Field
(GFF), a
random construction that can be considered a higher dimensional variant of Brownian motion.
Crucially, the law of the GFF is conformally invariant.
When $\gamma = \sqrt{8/3}$, the resulting random surface has long
been conjectured to be equivalent to the Brownian map. 
Very recently, work began to unify the two models and although the exact nature of this equivalency
(that is, a canonical embedding) is still 
unknown, major progress was made by Miller and Sheffield in their series of works
\cite{Miller1,Miller2,Miller3}; see also the recent survey \cite{MillerSurvey}.
In particular, they showed that the Brownian map and LQG share the same axiomatic properties and
that there exists a homeomorphism of the Brownian map onto $\bbS^2$ which realises LQG.
As above, the quasisymmetrically invariant starry property implies that no such embedding
can be quasisymmetric.
In \cite{Gwynne19}, Gwynne, Miller, and Sheffield gave an explicit construction of the embedding as the limit of
conformal embeddings, showing that the discretised Brownian disk converges to the conformal
embedding of the continuum Brownian disk. The latter corresponds to $\sqrt{8/3}$-LQG. 

Before describing our results and background in detail, we remark that the Brownian map
can also be obtained from another famous random space: the Brownian continuum
random tree.
The Brownian continuum random tree (CRT) is a continuum tree that was introduced and studied by Aldous in
\cite{Aldous1,Aldous2,Aldous3}. It appears in many seemingly disjoint contexts such as the scaling
limit of critical Galton-Watson trees and Brownian excursions using a ``least intermediate point''
metric. 
This ubiquity led to the CRT becoming an important object in probability theory in its own right. 
The Brownian map can be constructed from the CRT by another change in metric and it is this
description that allows access to a very different set of probabilistic tools on which our proofs
are based.
In fact, we first prove that the CRT is starry and thus also cannot be embedded into finite
dimensional manifolds almost surely. Note that this is also in stark contrast with the
work on conformal weldings of the CRT in \cite{Lin18} and implies much higher singularity of
embeddings than previously known.

This article is structured as follows: In Section \ref{sect:AssouadDim}, we recall the definition of
the Assouad dimension and its use in embedding theory. We will also introduce the \emph{starry}
property for metric spaces in Section \ref{sect:starrySpaces} and show that it is invariant under
quasisymmetric mappings and implies infinite Assouad dimension.
In Section \ref{sect:CRT}, we define the Brownian continuum random tree via the Brownian excursion and
show that the CRT is starry almost surely. We conclude that it cannot be embedded into $\R^d$ for
any $d\in\N$ using quasisymmetric mappings.
In Section~\ref{sect:BrownianMap}, we define the Brownian map through the CRT and prove that the
Brownian map is starry, also. In particular, Theorem \ref{thm:main} states that quasisymmetric
images of starry metric spaces have infinite Assouad dimension (and are thus not doubling), Theorem
\ref{thm:CRTStarry} proves that the CRT is starry almost surely, and Theorem \ref{thm:mainBM} shows
that the Brownian map is starry.
Section~\ref{sect:dualDiscussion} finishes this article by containing a discussion of our results
from a fractal geometric point of view. 
Throughout, we postpone proofs until the end of their respective section.

\section{Assouad dimension and embeddings}
\label{sect:AssouadDim}
The Assouad dimension is an important tool in the study of embedding problems. It was first
introduced by Patrice Assouad in \cite{AssouadThesis}.
More recently, the exact determination of the Assouad dimension for random and deterministic subsets
of Euclidean space has revealed intricate relations with separation properties in the study of
fractal sets. Notably, the Assouad dimension of random sets tends to be ``as big
as possible'', see e.g.\ \cite{FraserMiaoTroscheit}, and that for self-conformal subsets of $\R$
Ahlfors regular is equivalent to the Assouad dimension coinciding with the Hausdorff
dimension, see \cite{Angelevska18}.

\subsection{Assouad dimension}
\label{sect:AssouadSub}
Formally, let $(X,d)$ be a metric space and write $N(X,r)$ for the minimal number of sets of
diameter at most $r$ needed to cover $X$. We set $N(X,r)=\infty$ if no such collection exists. 

Let $B(x,r)$ be the closed ball in $X$ of radius $r>0$. The Assouad dimension is then given by 
\begin{multline}\label{eq:AssouadDefn}
  \dimA(X)= \inf \Bigg\{ \alpha \  : \ ( \exists  \, C>0) \, (\forall \, 0<r<R<1) \,
    \sup_{x\in X}  N\big( B(x,R), r \big) \ \leq \ C \left(\frac{R}{r}\right)^\alpha  \Bigg\}.
\end{multline}

There are several important generalisations and variations of the Assouad dimension such as the
quasi-Assouad dimension and the Assouad spectrum. The latter fixes the relationship between $r$ and
$R$ by a parameter $\theta$, i.e.\ $r=R^{1/\theta}$, whereas the former is a slightly more
regularised version of the Assouad dimension. 
One could certainly ask the questions that arise here of these variants and we forward the
interested reader to the survey~\cite{FraserSurvey} for an overview.

The Assouad dimension is always an upper bound to the Hausdorff dimension but coincides in many
``natural'' examples such as $k$-dimensional Riemannian manifolds. However, they can also differ
widely in general metric spaces and it is possible to construct a space $X$ such that $X$ is
countable with Hausdorff dimension $0$ but has infinite Assouad dimension; see e.g.\
Proposition~\ref{thm:notStarry}. 
These sets are however somewhat pathological and it would be interesting to find `natural' metric
spaces that have low Hausdorff and box-counting dimension but are `big' in the sense of Assouad
dimension. A natural candidate are random sets since they tend to have a regular average behaviour,
giving low Hausdorff dimension, but have rare but very `thick' regions; see
\cite{FraserMiaoTroscheit,FraserTroscheit,TroscheitRecursive,TroscheitThreshold}. These regions are detected
by the Assouad dimension and, as we will show in this article, are sufficient to give infinite Assouad
dimension for the Brownian map and CRT.

The Assouad dimension is strongly related to the metric notion of doubling: a metric space has
finite Assouad dimension if and only if it is doubling.
Further, the Assouad dimension is invariant under bi-Lipschitz mappings and is a useful indicator when a
space is or is not embeddable. Because of this invariance, a metric space $(X,d)$ with Assouad dimension
$s_a=\dimA X$ cannot be embedded into $\R^{\lceil s_a \rceil-1}$ with bi-Lipschitz mappings. The
converse is not quite true, but ``snowflaking'' the metric by some $\alpha>0$ allows a bi-Lipschitz
embedding~\cite{AssouadThesis}.
\begin{theorem}[Assouad Embedding Theorem]
  \label{thm:Assouad}
  Let $(X,d)$ be a metric space with finite Assouad dimension. Then there exists $C>1$, $N\in\N$, $1/2<\alpha
  <1$, and an injection $\phi:X\to\R^N$ such that
  \[
    C^{-1}d(x,y)^\alpha \leq |\phi(x)-\phi(y)| \leq C d(x,y)^\alpha\quad\forall x,y\in X
  \]
\end{theorem}
Explicit bounds on $N$ and $\alpha$ can be obtained from the Assouad dimension, see e.g.~\cite{David12}.

\subsection{Starry metric spaces}
\label{sect:starrySpaces}
While the Assouad dimension is invariant under bi-Lipschitz mappings, this is not true for the more
general notion of quasisymmetric mappings.
We recall the definition of a quasisymmetric mapping.

\begin{definition}
  Let $(X,d_X)$ and $(Y,d_Y)$ be metric spaces. A homeomorphism $\phi\from X \to Y$ is called
  \textbf{$\Psi$-quasisymmetric} if there exists an increasing function $\Psi\from(0,\infty)\to(0,\infty)$ such
  that for any three distinct points $x,y,z\in X$,
  \begin{equation*}
    \frac{d_Y(\phi(x),\phi(y))}{d_Y(\phi(x),\phi(z))} \leq \Psi\Big(\frac{d_X(x,y)}{d_X(x,z)}\Big).
  \end{equation*}
\end{definition}
A basic example of quasi-symmetric mappings are isometric embeddings but the notion of
quasisymmetric mappings generalises it by allowing a uniformly controlled distortion.
In Euclidean space quasisymmetric mappings correspond to quasi-conformal mappings. 
That is, if $\phi\from
\Omega\to\Omega'$, where $\Omega,\Omega'\subset \R^d$ are open and $\phi$ is $\Psi$-quasisymmetric, then $\phi$ is
also $K$-quasiconformal for some $K$ depending only on the function $\Psi$. A similar statement
holds in the other direction.

In \cite{MackayBook}, Mackay and Tyson study the Assouad dimension under symmetric mappings. In
particular, it is possible to lower the Assouad dimension by quasisymmetric mappings (see also
\cite{Tyson01}) and they introduce
the notion of the conformal Assouad dimension. The conformal Assouad dimension of a metric space
$(X,d)$ is defined as the infimum of the Assouad dimension of all quasisymmetric images of $X$.
This notion has been subsequently explored for many examples of deterministic sets, such as self-affine
carpets in \cite{Kaenmaki18}.

Here we define the structure of an approximate $n$-star, which, heuristically, is the property that a space contains
$n$ distinct points that are roughly equidistant to a central point with every geodesic between them
going near the centre, the extent of which is controlled by the parameters $A$ and $\mu$.
\begin{definition}
  A metric space $(X,d)$ is said to contain an \textbf{$(A,\eta)$-approximate $n$-star} if there
  exists $A>1$ and  $0<\eta<1$, $\varrho>0$, such that $A-\eta > 1+\eta \Leftrightarrow
  \eta<(A-1)/2$ as well as $\varrho>0$ and a set of points
  $\{x_0,\dots,x_n\}\subseteq(X,d)$ satisfying
  \[
    (A-\eta)\varrho\leq d(x_i,x_j)\leq A\varrho\text{ for all }i,j\in\{1,\dots,n\}
    \text{ with } i\neq j
  \]
  and
  \[
    \varrho\leq d(x_0,x_i)\leq (1+\eta) \varrho \text{ for all }i\in\{1,\dots,n\}.
    \]
    We say that a metric space $(X,d)$ is \textbf{starry} if there exist uniform $A>1$ and $0<\eta<(A-1)/2$ such that
    $(X,d)$ contains an $(A_n,\eta)$-approximate $n$-star for all $n$, where $A\leq A_n$.
\end{definition}
We note that any $(A,\eta)$-approximate $n$-star is also an $(A,\zeta)$-approximate $m$-star for all
$m\leq n$ and $\zeta\geq \eta$, provided that $1+\zeta < A-\zeta$.
Further, $A$ is always bounded above by the triangle inequality, since $(A-\eta)\varrho \leq
d(x_1,x_2)\leq d(x_1,x_0)+d(x_0,x_2)\leq 2(1+\eta)\varrho$, giving $1<1+2\eta<A\leq 2+3\eta<5$.

Our main theorem in this section states that all quasisymmetric images of starry metric spaces
(including the identity) have infinite Assouad dimension. Essentially, being starry means that the
conformal Assouad dimension is maximal.

\begin{theorem}\label{thm:main}
  Let $(X,d)$ be a starry metric space and 
  let $\phi$ be a quasisymmetric mapping $\phi:(X,d)\to(\phi(X),d_\phi)$.
  Then $\dimA\phi(X) = \infty$.
\end{theorem}
Recall that a metric space $(X,d)$ is said to be \textbf{doubling} if there exists a constant $K>0$
such that the ball $B(x,r)\subseteq X$ can be covered by at most $K$ balls of radius $r/2$ for all
$x\in X$ and $r>0$.
Given that the Assouad dimension of a metric space is finite if and only if it is
doubling, see e.g.~\cite[Theorem 13.1.1]{FraserBook}, we obtain
\begin{corollary}\label{thm:mainCor}
  Let $(X,d_X)$ be a starry metric space and $(Y,d_Y)$ be a doubling metric space. Then $X$ is not
  doubling and any embedding $\phi\from X\to Y$ cannot be quasisymmetric.
\end{corollary}
It is a simple exercise to see that any $s$-Ahlfors regular space\footnote{A metric space $X$ is $s$-Ahlfors
  regular if it supports a Radon measure $\mu$ such that $\mu(B(x,r))\sim r^s$ for all $x\in X$ and
$0<r<\diam X$.} has Assouad (and Hausdorff) dimension equal to $s$ and so $\R^d$ has Assouad
dimension $d$. Similarly, any $d$-dimensional Riemannian space is $d$-Ahlfors regular, as it supports a
$d$-regular volume measure.
Our result immediately implies that any starry metric space cannot be embedded into such finite
dimensional spaces by quasisymmetric (and hence bi-Lipschitz) mappings.

Careful observation of the estimates in the proof of Theorem~\ref{thm:main} shows that
snowflaking does not allow a quasisymmetric or
bi-Lipschitz embedding. That is, there is no analogy of the Assouad embedding theorem
(Theorem~\ref{thm:Assouad}) for starry
metric spaces and we get the stronger statement
\begin{corollary}
  Let $(X,d_X)$ be a starry metric space and let $(Y,d_Y)$ be a doubling metric space (such as an
  $s$-Ahlfors regular space with $s\in
  [0,\infty)$). Let $\alpha\in (0,1]$, $\Psi\from
  (0,\infty)\to(0,\infty)$ be an increasing function, and $\phi\from X\to Y$ be an embedding of $X$ into $Y$. Then
  $\phi$ cannot satisfy
  \begin{equation*}
    \frac{d_Y(\phi(x),\phi(y))}{d_Y(\phi(x),\phi(z))} \leq
    \Psi\Big(\frac{d_X(x,y)^{\alpha}}{d_X(x,z)^{\alpha}}\Big)
  \end{equation*}
  for distinct $x,y,z\in X$.
\end{corollary}

In analogy to the observation that every set $X\subset\R^d$ with Assouad dimension $d\in\N$ must
contain $[0,1]^d$ as a weak tangent, see \cite{FraserHanBLMS}, one might think that  
a metric space with infinite Assouad dimension must also be starry. However, that is not true.
\begin{proposition}\label{thm:notStarry}
  There exists a countable and bounded metric space with infinite Assouad dimension that is not starry.
\end{proposition}
We will give an example in the next section.

\subsection{Proofs for Section \ref{sect:AssouadDim}}
\begin{proof}[Proof of Theorem \ref{thm:main}]
  We argue by contradiction. Assume $\dimA \phi(X)<\infty$. Then there exists $s>0$ and $C>1$ such
  that $N(B_{d_\phi}(x,R),r) \leq C\cdot (R/r)^s$ for all $0<r<R<\diam_{d_\phi} \phi(X)$ and $x\in
  \phi(X)$.  We assume that $X$ is starry and thus there exist $A$ and $\eta$ such that $X$ has
  $(A,\eta)$-approximate $n_k$-stars.
  Pick $k$ such that $n_{k} > C\cdot (4\Psi(1)\Psi(1+\eta))^s$, where $\Psi:(0,\infty)\to[0,\infty)$ is the scale
  distortion function of the quasisymmetric mapping $\phi$.
  
  Let $x_i$, $\varrho$, be the points and size of the approximate $n_k$-star.
  First note that distances relative to the centre
  $x_0$ are preserved. That is, for all $i \in \{1,\dots,n_k\}$,
  \[
    \frac{d_\phi(\phi(x_0),\phi(x_i))}{d_\phi(\phi(x_0),\phi(x_j))}
    \leq \Psi\left(\frac{d(x_0,x_i)}{d(x_0,x_j)}\right)
    \leq \Psi\left(\frac{(1+\eta)\varrho}{\varrho}\right)\leq \Psi(1+\eta).
  \]
  One deduces a lower bound by taking the inverse 
  \[
    \frac{1}{\Psi(1+\eta)}\leq
  \frac{d_\phi(\phi(x_0),\phi(x_i))}{d_\phi(\phi(x_0),\phi(x_j))}\leq \Psi(1+\eta) 
  \]
  for all $i,j\in\{1,\dots,n_k\}$.
  Note that
  $B_{d_\phi}(\phi(x_0),R) = \{y\in \phi(X) \,: \, d_\phi(\phi(x_0),y)\leq R\}$ 
  contains $\{\phi(x_0),\dots,\phi(x_n)\}$ for  $R=\Psi(1+\eta)\min_k d_\phi(\phi(x_0),\phi(x_k))$.

  Let $i\neq j\in\{1,\dots,n_k\}$. We estimate
  \[
    \frac{d_\phi(\phi(x_0),\phi(x_i))}{d_\phi(\phi(x_i),\phi(x_j))}
    \leq \Psi\left(\frac{d(x_0,x_i)}{d(x_i,x_j)}\right) 
    \leq \Psi \left(\frac{1+\eta}{A-\eta}\right) \leq \Psi(1).
    \]
  and so obtain $d_\phi(\phi(x_i),\phi(x_j)) \geq d_\phi(\phi(x_0),\phi(x_i))/\Psi(1)\geq
  (1/\Psi(1))\min_k d_\phi(\phi(x_0),\phi(x_k))$.
  Set $r=(1/(4\Psi(1)))\min_kd_\phi(\phi(x_0),\phi(x_k))$ and let $y_i,y_j\in\phi(X)$ be such that $\phi(x_i)\in B_{d_\phi}(y_i,r)$
  and $\phi(x_j)\in  B_{d_\phi}(y_j,r)$ for $i\neq j \in\{1,\dots,n_k\}$.
  Then,
  \begin{align*}
    B_{d_\phi}(y_i,r) \cap B_{d_\phi}(y_j,r)
    &=\{z\in\phi(X)\,:\, d_\phi(z,y_i) < r \text{ and } d_\phi(z,y_j)<r\}\\
    &\subseteq \{z\in\phi(X)\,:\, d_\phi(z,\phi(x_i))<2r \text{ and }d_\phi(z,\phi(x_j))<2r\}\\
    &\subseteq \{z\in\phi(X)\,:\, d_\phi(\phi(x_i),z)+d_\phi(z,\phi(x_j)) <
    4r=(1/\Psi(1))\min_k d_\phi(\phi(x_0),\phi(x_i))\}\\
    &=\varnothing
  \end{align*}
  by the triangle inequality and our estimate\footnote{This also
  shows that the starry property is invariant under quasisymmetric mappings.} 
  for $d_\phi(\phi(x_i),\phi(x_j))$. 
  We conclude that any $r$-cover of $\{\phi(x_1),\dots, \phi(x_{n_k})\}$ must consist of at least
  $n_k$ balls as no single $r$-ball can cover two distinct points. Hence,
  \[N(B_{d_\phi}(\phi(x_0), R), r)\geq n_k\quad\text{ and }\quad
  \frac{R}{r}=4\Psi(1)\Psi(1+\eta).
  \]But then
  \[
    C (4\Psi(1)\Psi(1+\eta))^s < n_k \leq N(B_{d_\phi}(\phi(x_0), R), r)
    \leq C \Big(\frac{R}{r}\Big)^s = C (4 \Psi(1+\eta)\Psi(1))^s,
  \]
  a contradiction.
\end{proof}

\begin{proof}[Proof of Proposition~\ref{thm:notStarry}]
  We construct an explicit example of a metric space that is bounded, has infinite Assouad
  dimension, but is not starry.
  Consider the set of points $(n,m)\in\N\times\N$ with the pseudometric
  \[
    d((n,m),(n',m')) = \begin{cases} 0 &\text{if } n=n'\text{ and }m=m'\\
      0&\text{if } n=n', \min\{m,m'\}=1 \text{ and }\max\{m,m'\}>n\\
      0&\text{if } n=n' \text{ and } \min\{m,m'\}>n\\
      2^{-n}&\text{if }n=n' \text{ and none of the above}\\
      2\sum_{k=\min\{n,n'\}}^{\max\{n,n'\}}k^{-2}&\text{if } n\neq n'.
    \end{cases}
  \]
Since $d((n,m),(n,m)) = 0$ and the distance function is symmetric, we only need to check the
  triangle inequality.
  Let $(n,m),(n',m'),(n'',m'')$ be points in our space. If $n\neq n''$, then
  \begin{multline*}
    d((n,m),(n'',m''))=2\sum_{k=\min\{n,n''\}}^{\max\{n,n''\}}k^{-2}\leq
    2\sum_{k=\min\{n,n',n''\}}^{\max\{n,n',n''\}}k^{-2}\\
    \leq  2\sum_{k=\min\{n,n'\}}^{\max\{n,n'\}}k^{-2} +2\sum_{k=\min\{n',n''\}}^{\max\{n',n''\}}k^{-2}
    =d((n,m),(n',m'))+d((n',m'),(n'',m'')).
  \end{multline*}
  If, however $n=n''$, we may assume that $m\neq m''$ and $\min\{m,m''\}\leq n$ and
  $(\min\{m,m''\}\neq 1$ or $\max\{m,m''\}\leq n)$ as otherwise the triangle inequality is trivially satisfied.
  If $n'\neq n$, then 
  \[d((n,m),(n'',m'')) = 2^{-n} \leq 2 n^{-2} \leq
  2\sum_{k=\min\{n,n'\}}^{\max\{n,n'\}}k^{-2}\leq d((n,m),(n',m')).\]
  When $n'=n$ then $m'\neq m$ or $m'\neq m''$. Therefore, at least one of $d((n,m),(n',m'))$ and
  $d((n',m'),(n'',m''))$ is equal to $2^{-n}$. Again we obtain
  \[ d((n,m),(n'',m''))= 2^{-n} \leq d((n,m),(n',m'))+d((n',m'),(n'',m''))\]
  and $d$ is a pseudometric. In fact, identifying points with distance zero gives the metric space
  $((\N\times\N)/\hspace{-.2em}\approx, d)$ where all $(n,m)$ get identified with $(n,1)$ for $m>n$.
  
  To see that this space has infinite Assouad dimension one can consider balls $B_n$ centered at $(n,1)$
  with diameter $R=2^{-n}$. This ball contains exactly $n$ distinct points $(n,1),(n,2),\dots,(n,n)$ each
  at distance $R$ from each other. Letting $r=R/2$ we need $n$ balls of diameter $r$ to cover $B_n$.
  However, there are no uniform constants $C,s$ such that $n\leq C (2)^s$ and hence the Assouad
  dimension is infinite.
  
  To show that this metric space is not starry, we assume for a contradiction that it is and that
  there exist $1<A$, $\eta<(A-1)/2$ and subsets $S_n$ that form $(A_n,\eta)$-approximate
  $n$-stars for $A<A_n $.
  Consider $S_n$. It must contain a centre $x_0=(p_0,q_0)\in S_n$ and $n$ distinct points in the
  annulus $D_n=B(x_0,(1+\eta)\rho_n)\setminus B(x_0,\rho_n)$. We see that $(1+\eta)\rho_n > 2^{-n}$ since
  otherwise $S_n\subseteq B_n$ and all points in $B_n$ are equidistant and do not have a centre.
  We split the annulus $D_n$ in two parts, 
  \[
    D_n^- = \{(p,q)\in\N : p<p_0 \text{ and }
  \rho_n \leq \sum_{k=p}^{p_0}k^{-2}\leq (1+\eta)\rho_n \}
\]
and  
\[
  D_n^+ = \{(p,q)\in\N : p>p_0 \text{ and }
  \rho_n \leq \sum_{k=p_0}^{p}k^{-2}\leq (1+\eta)\rho_n \}
\]
Let $(p,q),(p',q')\in D_n^+$ be distinct. Then, 
\[
  d((p,q),(p',q'))\leq
  \max\{2^{-(n+1)},\sum_{k=p_0}^{p}k^{-2}\}\leq (1+\eta)\rho_n < (A-\eta)\rho_n \leq
  (A_n-\eta)\rho_n.
\]
Thus, $S_n$ may contain at most one element in $D_n^+$ and $D_n^-$ contains at least $n-1$ elements.
Consider the distinct elements $(p,q),(p',q')\in D_n^-$. If $p\neq p'$ we must have
$d((p,q),(p',q')) \leq \sum_{k=p_0}^{p}k^{-2} \leq(1+\eta)\rho_n < (A_n-\eta)\rho_n$. This implies
that at least $n-1$ elements in $D_n^-$ are contained a single $B_p\subset D_n^-$. Considering
distinct elements $(p,q),(p',q')\in D_n^-$ with
$p=p'$, we obtain $d((p,q),(p',q'))=2^{-p}$, where $p$ satisfies $\sum_{k=p}^{p_0}k^{-2}\leq
(1+\eta)\rho_n$. Since these points form an approximate $n$-star, we further have
$2^{-p}>(A_n-\eta)\rho_n>(1+\eta)\rho_n$ and so
\[
  p^{-2}\leq \sum_{k=p}^{p_0}k^{-2} \leq(1+\eta)\rho_n < 2^{-p}.
\]
This implies $p\leq 4$ and so $\# S_n \leq 5$, a contradiction as $n$ was arbitrary.

  Lastly, the space is bounded as the entire set is contained in $B((1,1),2\pi^2/6)$.
\end{proof}

\section{Result for Brownian Continuum Random Trees}
\label{sect:CRT}

In this section we introduce the concept of $\R$-trees. We define a pseudometric on $[0,1]$ 
in terms of an excursion function $f$. This pseudometric gives
rise to an $\R$-tree which we call the continuum tree of function $f$. Letting $f$ be a generic
realisation of the Brownian excursion, we obtain the Brownian continuum random tree (CRT).
We will show that the CRT is a starry metric space and thus has infinite Assouad dimension.
In Section \ref{sect:dualDiscussion} we ask whether this is true for a larger class of functions.

\subsection{\texorpdfstring{$\R$}{R}-trees and excursion functions}
\label{sect:Excursions}
An {$\R$-tree} is a continuum variant of a tree. It is a metric space that satisfies the following
properties.
\begin{definition}
  A metric space $(X,d)$ is an \textbf{$\R$-tree} if, for every $x,y\in X$,
  \begin{enumerate}
    \item there exists a unique isometric mapping $f_{(x,y)}\from [0, d(x,y)]\to X$ such that
      $f_{(x,y)}(0)=x$ and $f_{(x,y)}(d(x,y)) = y$,
    \item if $f\from [0,1]\to X$ is injective with $f(0)=x$ and $f(1)=y$, then
      \[f([0,1]) = f_{(x,y)}([0,d(x,y)]).\]
   \end{enumerate}
   We further say that $(X,d)$ is \textbf{rooted} if there is a distinguished point $x\in X$, which we call
   the root.
\end{definition}
Heuristically, $X$ is a connected, but potentially uncountable, union of line segments. Every two points $x,y\in X$ are connected
by a unique arc (or geodesic) that is isomorphic to a line segment. We call any $x\in X$ a leaf if
$X\setminus\{x\}$ is still connected. In Section~\ref{sect:BrownianMap} we will further introduce a
labelling on the trees that is used to identify (or glue) certain leaves. The resulting space will be
the Brownian map.

There are several canonical ways of generating $\R$-trees, such as the "stick-breaking model" but
here we will focus on continuum trees generated by excursion functions.
\begin{definition}
  Let $f\from[0,1]\to[0,\infty)$ be a continuous function. We say that $f$ is a 
  \textit{\textbf{(length $1$) excursion function}} if $f(0)=f(1)=0$ and $f(t)>0$ for all $t\in(0,1)$. 
\end{definition}
These excursion functions are now used to define a new metric on $[0,1]$ that gives a new metric
space called the continuum tree.
\begin{definition}\label{def:contTree}
  Let $f$ be an excursion function. We call $T_f=([0,1]/\hspace{-.2em}\approx,d_f)$ the
  \textbf{continuum tree} with excursion $f$,
  where $d_f$ is the (pseudo-)metric given by
  \[
    d_f(x,y)=f(x) + f(y) -2 \min\{f(t) \;:\; \min\{x,y\} \leq t \leq \max\{x,y\}\,\},
  \]
  and $\approx_f$, the equivalence relation on $[0,1]$ defined by $x\approx_f y$ if and only if $d_f(x,y)=0$.
\end{definition}
The resulting metric space is of cardinality the continuum and can be considered a tree with root
vertex $0\approx_f 1$, where the lowest common ancestor of the equivalence classes of $x < y$ is given
by the equivalence classes of any $t_0\in[x,y]$ such that $f(t_0)=\min\{f(t) \;:\; x\leq t \leq y\}$.
We note that while the value for which the minimum is achieved might not be unique in $[0,1]$ with
the Euclidean metric, all such points are identified in the metric space
$T_f=([0,1]/\hspace{-.2em}\approx_f ,d_f)$, where $d_f$ is a bona fide metric.

\subsection{The Brownian Continuum Random Tree}
\label{sect:CRTStarry}
The Brownian continuum random tree was first studied in the comprehensive work of Aldous \cite{Aldous1, Aldous2, Aldous3}.
It is a random metric space that can be obtained with the continuum tree metric described above by
choosing the normalised Brownian excursion $\ex(t)$ as the excursion function.
The Brownian excursion can be defined in terms of a Brownian bridge $B(t)$ with parameter $T$, which is a Wiener process
conditioned on $B(0)=B(1)=0$. Further, it is well known that the Brownian bridge (with $T=1$) can be expressed as
$B(t) = W(t) - t W(1)$ where $B(t)$ is independent of $W(1)$. The graph $\{(t,W(t)) : t\in[0,1]\}$
and hence $\{(t,B(t)) : t\in[0,1]\}$ are compact almost surely and so
$B(t_{\min})=\min\{B(t)\;:\;0\leq t\leq 1\}$ exists and $t_{\min}$ is almost surely unique.
Cutting the Brownian bridge at the minimum and translating, one obtains a Brownian excursion
\[
  \ex(t) = \begin{cases}
    B(t+t_{\min})-B(t_{\min}) & 0 \leq t \leq 1-t_{\min},\\
    B(t-1+t_{\min})-B(t_{\min}) & 1-t_{\min} < t \leq 1.
  \end{cases}
\]
We use this definition of the Brownian bridge and excursion, as we will need to show a decay of
correlations between $\ex(s)$ and $\ex(t)$, where $s$ and $t$ are in disjoint subintervals of
$[0,1]$. 
\begin{definition}
  Let $\ex$ be a Brownian excursion. The random metric space $(T_{\ex},d_{\ex}) =
  ([0,1]/\hspace{-.2em}\approx_{\ex},d_{\ex})$ is called
  the \textbf{Brownian continuum random tree (CRT)}.
\end{definition}
The CRT also appears as the limit object of the stick-breaking model and rescaled critical
Galton-Watson trees as the number of nodes is taken to infinity. As such, the CRT is an important
object in probability theory. It also appears in the construction of the Brownian map, which we will
recall in Section~\ref{sect:BrownianMap}.

Our main result in this section is that the CRT is starry. However, we establish a slightly stronger
result below, which we will need in the proof that the Brownian map is starry.
\begin{theorem}
  \label{thm:CRTStarry}
  Let $\ex$ be a Brownian excursion and $T_{\ex}$ be the associated Brownian continuum random tree.
  Then, for every $n$, $T_{\ex}$ contains infinitely many approximate $n$-stars, almost surely. In
  particular, $T_{\ex}$ is almost surely starry.
\end{theorem}
Therefore, by Theorem~\ref{thm:main} and Corollary~\ref{thm:mainCor}, the Brownian excursion has
infinite Assouad dimension and cannot be embedded into finite dimensional manifolds using 
quasisymmetric mappings.

We end this section by noting that, from a fractal geometry standpoint,
the continuum tree metric could hold the key to a better understanding of ``fractal functions'' such
as the Weierstra\ss\ functions and general self-affine functions, see Section~\ref{sect:dualDiscussion}.

\subsection{Proofs of Section~\ref{sect:CRT}}
To prove that $T_{\ex}$, and later that the Brownian map, is starry, we partition $[1/2,3/4)$ into
countably many disjoint intervals $A_n^m$ and show that the processes on these intervals are
``almost'' independent.
Let $A_n^k=[a(n,k),a(n,k)+2^{-(n+k+2)}]$, where $a(n,k)=\tfrac34-2^{-(n+k+1)}(2^{k-1}+1)$.
Noting that $a(n,k+1)-a(n,k)=2^{-(n+k+2)}$ and $a(n+1,1)>a(n,k)$ for all $n,k\in\N$ we see that the
interiors of $A_n^k$ and $A_{n'}^{k'}$ are disjoint whenever $n\neq n'$ or $k'\neq k$.
Further, $\bigcup_{n,k\in\N}A_n^k=[\tfrac12,\tfrac34)$ and the half-open intervals 
$[\inf A_n^k, \sup A_n^k)$ form a countable partition of
$[\tfrac12,\tfrac34)$.
We require the following well-known result that allows us to estimate a Wiener process with a given
function. In particular, this follows directly from the construction of the classical Wiener
measure on the classical Wiener space and the fact that this measure is strictly
positive, see e.g.~\cite[\S8]{StroockBook}.

\begin{lemma}\label{thm:approximation}
Let $W(t)$ be a Wiener process on $[0,1]$. Then, for every continuous function $f\in C[0,1]$ and
$\eps>0$, 
\[\Prob(\|W(t)+f(0)-f(t)\|_\infty < \eps) >0.\]
\end{lemma}
We will use this lemma with a ``zig-zag'' function (see Figure~\ref{fig:approximation}) that will
give the correct structure on $[0,1]$
when applied with the continuum tree metric.
Let $F_n(x):[0,1]\to \R_{\geq 0}$, where
\[
  F_n(x) = \begin{cases}
    2nx, & 0\leq x < 1/(2n);\\
    -nx+\frac{2k+3}{2}, & \frac{2k+1}{2n}\leq x < \frac{k+1}{n} \text{ and }0\leq k \leq n-2;\\
    nx-k-\tfrac12, & \frac{k+1}{n}\leq x < \frac{2k+3}{2n}\text{ and }0\leq k \leq n-2;\\
    -2nx+2n, & \frac{2n-1}{2n}\leq x \leq 1.
  \end{cases}
\]	
The function $F_n$ is an excursion function ($F_n(0)=0=F_n(1)$) that is continuous and linear
between the local maxima $F_n((k+1)/n)=1$
and local minima $F_n((2k+1)/(2n))=1/2$ in $(0,1)$.
We next show that the Brownian excursion contains arbitrarily good approximations to this function for all $n$.

\begin{lemma}\label{thm:excursionApprox}
  Let $\ex(s)$ be a Brownian excursion. Then, almost surely, for every
  $n$ large enough, there exists infinitely many pairwise disjoint non-trivial 
  intervals $[s_{n}^k, t_n^k]\subset [0,1]$, $k\in\N$, such that
  \[\sup_{s\in[s_n^k,t_n^k]}\lvert \ex(s)-\ex(s_n^k) - \sqrt{t_n^k-s_n^k}\cdot F_n(s) \rvert <
  \sqrt{t_n^k-s_n^k}\cdot 2^{1-n}.\]
\end{lemma}

\begin{proof}
  Using Lemma~\ref{thm:approximation}, we have $\|W(t)-F_n(t)\|_\infty<2^{-n}$ with positive
  probability, see Figure~\ref{fig:approximation}.
\begin{figure}[htb]
  \includegraphics[width=0.7\textwidth]{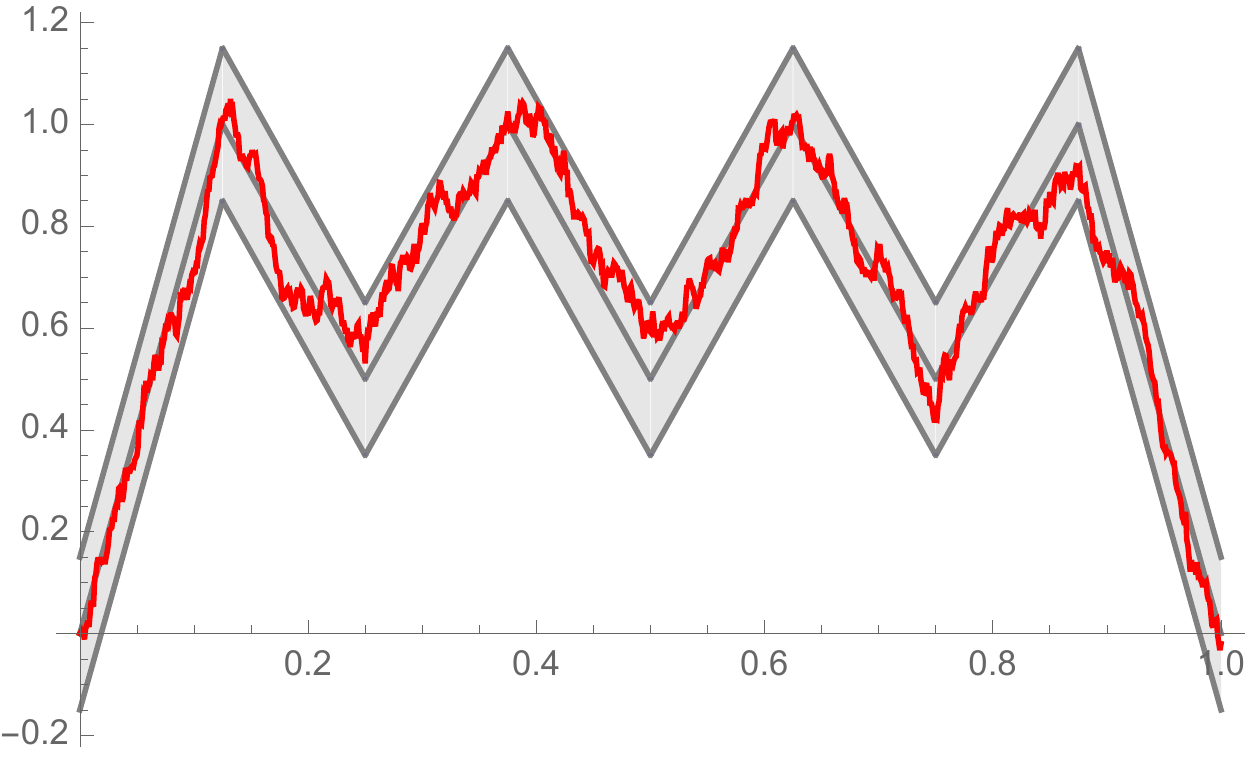}
  \caption{The function $F_4$ and a Wiener process such that $\|W-F_4\|_\infty < 1/8$.}
  \label{fig:approximation}
\end{figure}
Let $H_n^k:[0,1]\to A_n^k$ be the map defined by
$x\mapsto \lvert A_n^k\rvert x+\inf A_n^k$.
By the scaling property of Wiener processes, 
$W_n^k(s)=(W\circ H_n^k(s) - W\circ H_n^k(0))/\sqrt{\lvert A_n^k\rvert}$ is equal in
distribution to $W(s)$. Therefore, $\Prob \{\lVert W_n^k - F_n\rVert_\infty< 2^{-n}\}\geq p_n>0$,
where $p_n$ only depends on $n$. 
Further, the disjointness of intervals $A_n^k$ and $A_{n'}^{k'}$ for
$(n,k)\neq(n',k')$ implies that the events are independent for different $(n,k)$.
A simple application of the Borel-Cantelli lemma then shows that for every $n$ there are
infinitely many $k$ such that $\lVert W_n^k - F_n\rVert_\infty< 2^{-n}$ almost surely.

  Let us define the Brownian bridge $B$ by $B(s) = W(s)-sW(1)$ and let
\begin{align*}
  B_n^k(s) &= (B\circ H_n^k(s)-B\circ H_n^k(0))/\sqrt{\lvert A_n^k\rvert}\\ &= (W\circ H_n^k(s) - s W(1)\cdot
H_n^k(s)-W\circ H_n^k(0) + s W(1)\cdot
H_n^k(0))/\sqrt{\lvert A_n^k\rvert}.
\end{align*}
While it is not true that
$B_n^k(s)$ is independent of $B_{n'}^{k'}(t)$, the problematic $W(1)$ plays a diminishing role in
approximating $B_n^k$ by $F_n^k$ for $n,k$ large. 

Fix a realisation such that for
all $n$ there are infinitely many $k$ such that  $\lVert W_n^k - F_n\rVert<2^{-n}$.
Clearly this realisation is generic as the event has full measure.
Then, for this realisation,
\begin{align*}
  &\lVert B_n^k  - F_n\rVert_\infty\\=& \left\lVert \left(W\circ H_n^k -s\cdot  W(1)\cdot H_n^k-W\circ
  H_n^k(0)+s\cdot W(1)\cdot H_n^k(0)\right)/\sqrt{\lvert A_n^k\rvert}-F_n\right\rVert_\infty\\
  \leq&\left\lVert W_n^k - F_n\right\lVert_\infty + \left\lVert s\cdot W(1)\cdot(H_n^k(0) -
  H_n^k)\right\rVert_\infty / \sqrt{\lvert A_n^k\rvert}\\
  \leq& \;2^{-n} +\sqrt{\lvert A_n^k\rvert}\cdot \lvert W(1)\rvert
  = 2^{-n}+2^{-(n+k)/2-1}\lvert W(1)\rvert.
\end{align*}
Since $W(1)$ is fixed for this realisation  we obtain that
for every $n\in\N$ there are infinitely many $k\in\N$ such that $\lVert B_n^k - F_n\rVert <
2^{-(n-1)}$.
Finally, $\ex$ is a simple cut and translate transformation of $B$, where the cut is the
(almost surely) unique $t_{\min}$ with $B(t_{\min})=\min_{t\in[0,1]} B(t)$.
There are two cases to consider. Either $t_{\min} \geq 3/4$ or $t_{\min}<3/4$.
In the former case we define the interval $I_n^k=A_n^k+1-t_{\min}$ and, in the latter case, we define $I_n^k =
A_n^k-t_{\min}$ for $n$ large enough such that $\inf A_n^1 > t_{\min}$. Analogous to $H_n^k$, define
$G_n^k:[0,1]\to I_n^k$ to be the unique orientation preserving similarity mapping $[0,1)$ into
$I_n^k$.
Therefore, almost surely, for large enough $n$ there exist infinitely many $k$ such that
$\sup_{s\in [0,1)}\lvert (\ex\circ G_n^k(s)-\ex\circ G_n^k(0))/\sqrt{\lvert I_n^k\rvert} - F_n(s)\rvert<2^{1-n}$.
This is the required conclusion.
\end{proof}

Equipped with this lemma, we prove that the CRT is starry.

\begin{proof}[Proof of Theorem~\ref{thm:CRTStarry}]
Let $G_n^k$ be as in the proof of Lemma~\ref{thm:excursionApprox} and assume that $\ex$ is a generic
realisation. Then, 
\begin{equation}\label{eq:excursionEstimate}
  \ex\circ G_n^k(0)+\sqrt{\lvert I_n^k\rvert}\cdot(F_n(s)-2^{1-n})
  \leq \ex\circ G_n^k(s) \leq \ex\circ G_n^k(0)+\sqrt{\lvert I_n^k\rvert}\cdot(F_n(s)+2^{1-n})
\end{equation}
for all $0\leq s \leq 1$, where the existence of $k$ for large enough $n$ is guaranteed by
Lemma~\ref{thm:excursionApprox}.
Let $z = 1/(4n)$. Further, let $y_p = (2p+1)/2n$ for $0\leq p \leq n-2$. 
We now estimate the distances between $G_n^k(z)$ and $G_n^k(y_p)$ with respect to the $d_{\ex}$ metric.
First, by \ref{eq:excursionEstimate}, 
\begin{align*}
  d_{\ex}(G_n^k(z),G_n^k(y_p)) &= \ex(G_n^k(z))+\ex(G_n^k(y_p)) - 2 \min_{t\in [z,y_p]}\ex(G_n^k(t))\\\
  &\leq \sqrt{\lvert I_n^k\rvert} \left(F_n(z)+F_n(y_p) +2\cdot 2^{1-n} - 2 \min_{t\in [z,1-z]}
  (F_n(t)-2^{1-n})\right)\\
  & = \sqrt{\lvert I_n^k\rvert} (2^{3-n}+1/2+1-2\cdot 1/2) = \sqrt{\lvert I_n^k\rvert}
  \left(  \frac{1}{2}+2^{3-n}\right)
\end{align*}
and
\begin{align*}
  d_{\ex}(G_n^k(z),G_n^k(y_p)) 
  &\geq \sqrt{\lvert I_n^k\rvert} \left(F_n(z)+F_n(y_p) -2\cdot 2^{1-n} - 2 \min_{t\in [z,1-z]}
  (F_n(t)+2^{1-n})\right)\\
  & = \sqrt{\lvert I_n^k\rvert} (-2^{3-n}+1/2+1-2\cdot 1/2) = \sqrt{\lvert I_n^k\rvert}
  \left(  \frac{1}{2}-2^{3-n}\right).
\end{align*}
Similarly, for $q\neq p$,
\begin{align*}
  d_{\ex}(G_n^k(y_q),G_n^k(y_p)) &= \ex(G_n^k(y_q))+\ex(G_n^k(y_p)) - 2 \min_{t\in [y_q,y_p]}\ex(G_n^k(t))\\\
  &\leq \sqrt{\lvert I_n^k\rvert} \left(F_n(y_q)+F_n(y_p) +2\cdot 2^{1-n} - 2 \min_{t\in [z,1-z]}
  (F_n(t)-2^{1-n})\right)\\
  & = \sqrt{\lvert I_n^k\rvert} (2^{3-n}+1+1-2\cdot 1/2) = \sqrt{\lvert I_n^k\rvert}
  \left(  1+2^{3-n}\right)
\end{align*}
and
\begin{equation*}
  d_{\ex}(G_n^k(y_q),G_n^k(y_p))\geq \sqrt{\lvert I_n^k\rvert}\left(  1-2^{3-n}\right).
\end{equation*}
Let $\varrho_n = (1-2^{4-n}) \tfrac12 \sqrt{\lvert I_n^k\rvert}$, $A_n = (2+2^{4-n})/(1-2^{4-n})$,
and $\eta_n =
(1+2^{4-n})/(1-2^{4-n})-1$. For $n\geq6$ we obtain $A_n > 2$, $0<\eta_n\leq 2/3$, and 
\[
  A_n-\eta_n = (2-2^{4-n})/(1-2^{4-n})
> (1+2^{4-n})/(1-2^{4-n}) = 1+\eta_n.
\]
Therefore the points $x_0=G_n^k(z)$, $x_p=G_n^k(y_p)$ $(1\leq p
\leq n)$ form an $(A_n,\eta)$-approximate $(n-1)$-star  for sufficiently large $n$ and so, as $A_n
>2>1$, we conclude that $T_{\ex}$ is starry almost surely.
\end{proof}

\section{The Brownian Map}
\label{sect:BrownianMap}
As mentioned in the introduction, the Brownian map is a model of random geometry of the sphere that
arises as the limit of many models of planar maps, e.g.\ uniform $p$-angulations $(p\geq 3)$ of $\bbS^2$, as the number of
vertices is increased. We refer the reader to the surveys \cite{LeGall14} and \cite{MillerSurvey}
for a detailed description of the Brownian map and its relation to other random geometries, such as
Liouville Quantum Gravity.

In this section we define the Brownian map in terms of the Brownian continuum random tree and show
that the Brownian map is almost surely starry.

\subsection{Definition of the Brownian map}
\label{sect:BMDefn}
To obtain the Brownian map from the CRT, one defines a random pseudometric on the CRT, which in turn
results in a (quotient) metric space that is the Brownian map. This metric is defined in terms of a
Gaussian process defined on the CRT.

Let $Z(t)$ be a centered Gaussian process conditioned on $\ex$ such that the covariance satisfies
\[\E(Z(s)Z(t)|\ex) = \min\{\ex(u)\;:\;\min\{s,t\}\leq u \leq \max\{s,t\}\}.\]
Note that this completely determines the Gaussian process and that the second moment of distances 
is equal to the distance on the CRT, that is $\E((Z(s)-Z(t))^2|\ex) =d_{\ex}(s,t)$.
One may imagine this process as a Wiener process along geodesics of the CRT, where branches evolve
independently up to the common joint value.

We use this process to define a pseudo metric on $[0,1]$ by first setting 
\[
  D^o(s,t) = Z(s)+Z(t) - 2 \max\left\{ \min_{u\in[s,t]}Z(u), \min_{u\in[t,s]}Z(u) \right\}
\]
for $s,t\in[0,1]$.
We can also define $D^o$ directly on the tree $T_{\ex}$ by defining 
\[
  D^o(x,y) =
  \min\{D^o(s,t)\,:\,\pi(s)=x \text{ and } \pi(t)=y\}
\]
for $x,y\in T_{\ex}$,
where $\pi:[0,1]\to T_{\ex}$ is the canonical projection under the equivalence relation
$\approx_{\ex}$, cf.~Definition~\ref{def:contTree}.
Note that $D^o$ does not satisfy the triangle inequality and for $s,t\in[0,1]$ we
further define
\[
  D(s,t)= \inf\left\{ \sum_{i=0}^{n-1} D^o (u_i,u_{i+1}) : u_i \in T_{\ex} \right\},
\]
where the infimum is taken over all finite choices of $u_i$ satisfying $\pi(s)=u_0$ and $\pi(t)=u_n$. 
It can be verified that $D$ is indeed a pseudometric on $[0,1]$. Considering
$\M=[0,1]/\hspace{-.2em}\approx$,
where $s\approx t$ if and only if $D(s,t)=0$ we obtain the Brownian map $(\M,D)$.
More generally, we can construct a Brownian map for any excursion function $f$ by replacing $\ex$
with $f$ in the construction above. We call the resulting metric space $(\Phi(f),D)$ the Brownian
map conditioned on the excursion function $f$. Letting $f=\ex$, we recover the Brownian map
$(\M,D)=(\Phi(\ex),D)$.

\subsection{The Brownian map is starry.}
Lemma~\ref{thm:excursionApprox} gives an insight in the extremal structure of $T_{\ex}$ and we
will see that it also implies extremal behaviour for the Brownian map.
The approximate star for the CRT (see Figure~\ref{fig:Tree} gives rise to a line segment of
length approximately $\varrho$ to which $n-1$ line segments of length approximately $\varrho$ are
glued near the end.

By a similar argument to the one in Lemma \ref{thm:excursionApprox} we will see that there exists
positive probability that the Gaussian process defined on this substructure follows a similar ``zig-zag''
pattern that gives rise to a further approximate $(n-1)$-star, see Figure~\ref{fig:BrownianOnTree}.
\begin{figure}
  \begin{center}
  \includegraphics[width=.45\textwidth]{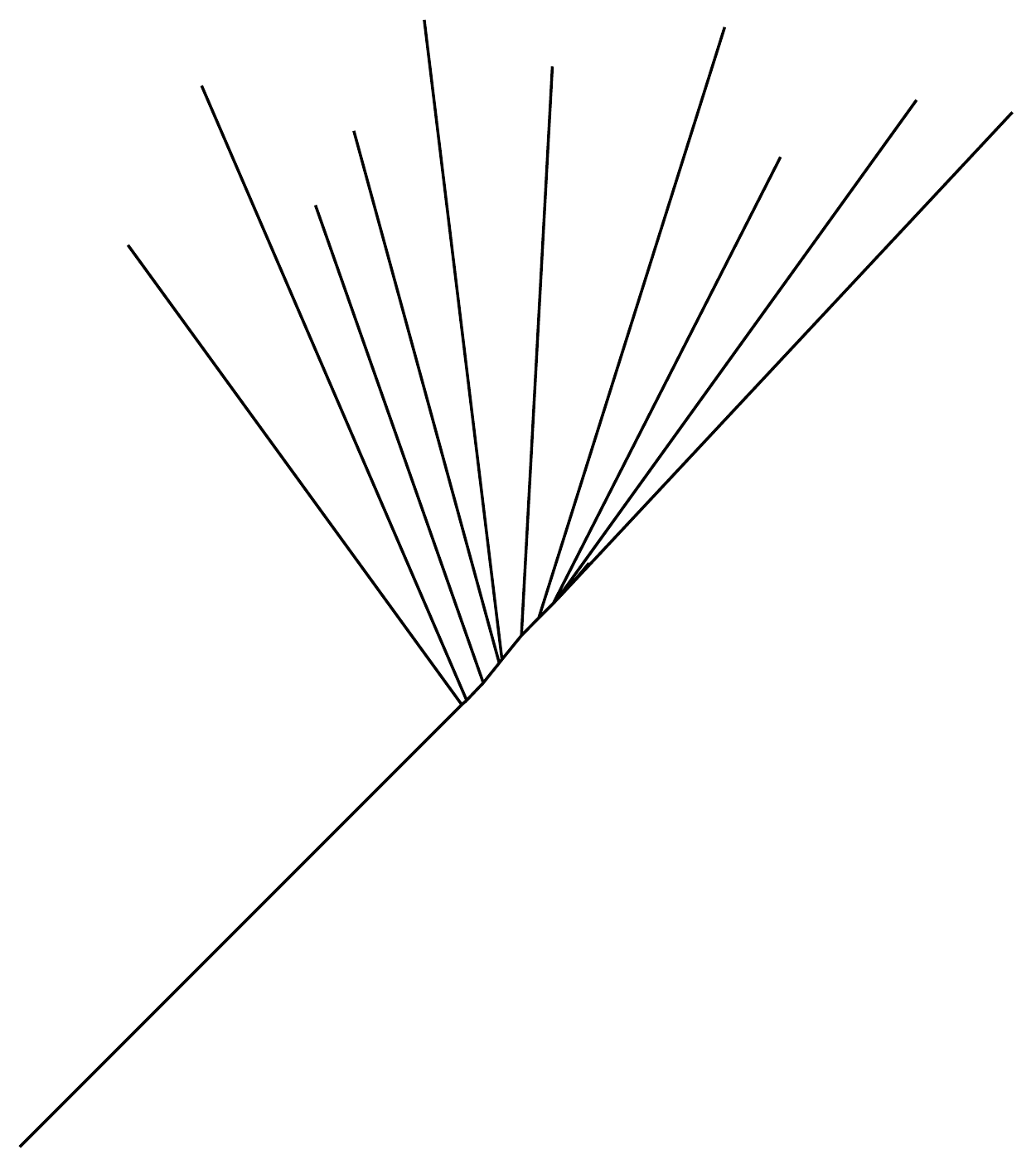}
\end{center}
\caption{The metric structure of an $F_{10}$ approximation.}
\label{fig:Tree}
\end{figure}

\begin{theorem}\label{thm:mainBM}
  Let $f:[0,1]\to \R$ be an excursion function such that for infinitely many $n\in\N$, there exists infinitely many intervals $[s_{n}^k, t_n^k]\in [0,1]$, $k\in\N$, such that
  \begin{equation}\label{eq:mapAssumptions}
    \sup_{s\in[s_n^k,t_n^k]}\lvert f(s)-f(s_n^k) - \sqrt{t_n^k-s_n^k}\cdot F_n((s-s_n^k)/(t_n^k-s_n^k)) \rvert <
  \sqrt{t_n^k-s_n^k}\cdot 2^{1-n}
  \end{equation}
  where the intervals $[s_n^k,t_n^k]$ are pairwise disjoint. Then, $(\Phi(f),D)$ is almost surely starry.
\end{theorem}
In particular, this is true for almost every realisation of the Brownian excursion, see
Lemma~\ref{thm:excursionApprox}.
\begin{corollary}
  The Brownian map $(\M,D)$ is almost surely starry and thus for all $0<\alpha\leq 1$, doubling
  metric spaces $(X,d)$ and increasing
  functions $\Psi\from (0,\infty)\to(0,\infty)$ any homeomorphism $\phi\from (\M,D)\to
  (X,d)$ cannot satisfy
  \[
    \frac{d(\phi(x),\phi(y)}{d(\phi(x),\phi(z))} \leq \Psi\left(
    \frac{D(x,y)^{\alpha}}{D(x,z)^{\alpha}} \right)
  \]
  for all $x,y,z\in\M$.
\end{corollary}
Note further that the intervals at which we chose to ``zoom in'' on the CRT and Brownian map are
arbitrary. We can easily make the stronger statement that there cannot be a single neighbourhood at which the
homeomorphism can be quasisymmetric, i.e.\ the map cannot even be locally quasisymmetric.

\subsection{Proofs of Section \ref{sect:BrownianMap}}
\begin{proof}[Proof of Theorem~\ref{thm:mainBM}]
  Temporarily fix $n,k$ such that \eqref{eq:mapAssumptions} holds. 
  Define 
  \begin{multline}\label{eq:x0def}
    x_0 =x_n= \argmax\Big\{\min\{f(s): 0 \leq s-s_n^k \leq (t_n^k-s_n^k)/(2n)\},\\\min\{f(s):
  (1-1/(2n))(t_n^k-s_n^k)\leq s-s_n^k \leq (t_n^k-s_n^k)\}\Big\},
\end{multline}
  \[x_p = \argmin\left\{f(s): \frac{p-1/2}{n}(t_n^k-s_n^k) \leq s-s_n^k \leq
  \frac{p+1/2}{n}(t_n^k-s_n^k)\right\}\quad\text{for}\quad 1\leq
  p\leq n-1,\]
  and 
  \[y_p = \argmax\left\{f(s): \frac{p}{n}(t_n^k-s_n^k) \leq
  s -s_n^k \leq \frac{p+1}{n}(t_n^k-s_n^k)\right\}\quad\text{for}\quad 0\leq p\leq n-1.\]
  Consider the ``subtree'' $([s_n^k,t_n^k],d_f)\subset (T_f,d_f)$ and note that it contains $n+1$ lines
  $L_i$ of length comparable to $(1/2)\sqrt{t_n^k-s_n^k}$. They are
  \[L_i = \{\max \{s\in[x_i,y_i]:f(s) = t \} \;:\; t\in [\max\{f(x_i),f(x_{i+1})\},f(y_i)]\}\]
  and we have $f(L_i) = [\max\{f(x_i),f(x_{i+1})\},f(y_i)]$, $\sup L_i \leq \inf L_j$ for $i< j$,
  $\min_{s\in L_i}f(s) = f(\inf L_i)$, and
  $f\rvert_{L_i}$ is a bijection.
  This structure can easily be seen on Figure~\ref{fig:BrownianOnTree} (left). The point $x_n=x_0$
  corresponds to the root
  of the tree, the $x_i$ correspond to the branch points of line segment $L_i$ that end at the leaf $y_i$.
  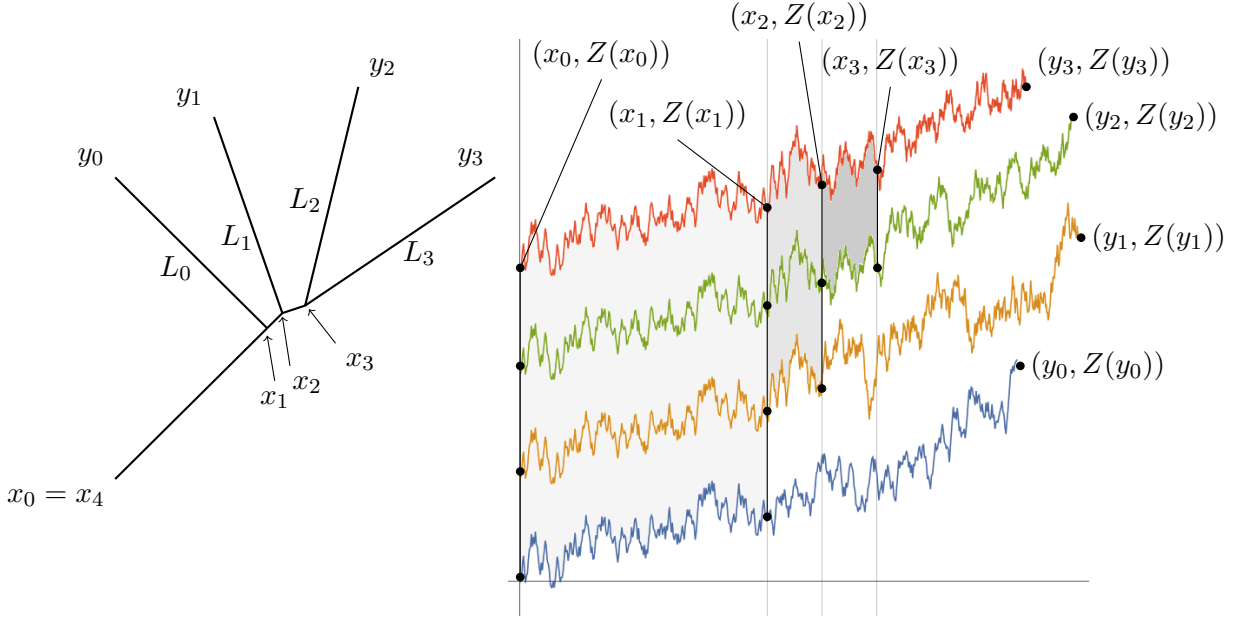
\begin{figure}
  \begin{tikzpicture}
    \draw[thick] (0,0) -- (2,2);
    \draw[->] (2.1,1.3) -- (2.0,1.9);
    \draw (2.1,1.3) node[anchor = north]{$x_1$};
    \draw (0,0) node[anchor = north east]{$x_0=x_4$};
    \draw[thick] (2,2) -- (0,4) node[anchor = south east]{$y_0$};
    \draw[thick] (2,2) -- (2.2,2.2) -- (1.3,4.8) node[anchor = south east]{$y_1$};
    \draw[->] (2.3,1.5) -- (2.2,2.15) ;
    \draw (2.2,1.5) node[anchor = north west]{$x_2$};
    \draw[->] (2.9,1.8) -- (2.55,2.25);
    \draw (2.9,1.8) node[anchor = north west]{$x_3$};
    \draw[thick] (2.2,2.2) -- (2.5,2.3) -- (3.2,5.2) node[anchor = south west]{$y_2$};
    \draw[thick] (2.5,2.3) -- (5.0,4) node[anchor= south east]{$y_3$};
    \draw (.8,2.8) node {$L_0$};
    \draw (1.6,3.2) node {$L_1$};
    \draw (2.5,3.7) node {$L_2$};
    \draw (4,3) node {$L_3$};
    \draw (9,2) node {\includegraphics[width=20em]{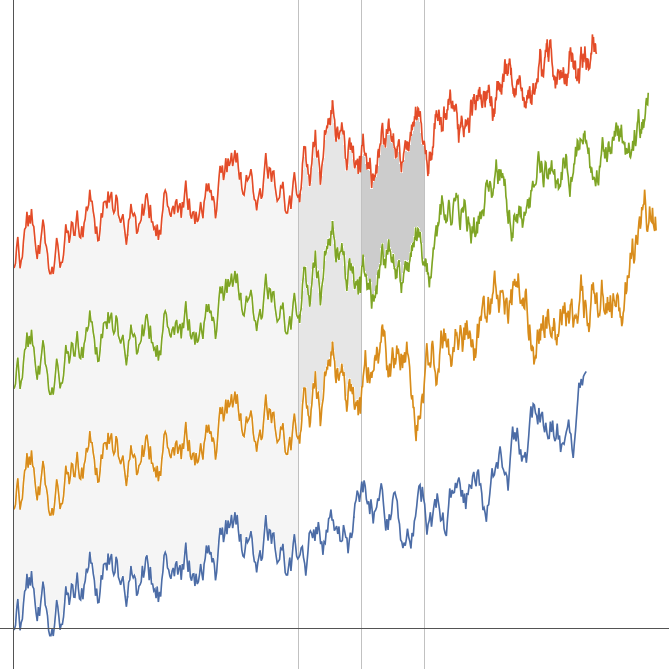}};
    \draw[fill=black] (5.33,2.8) circle (.05);
    \draw[fill=black] (5.33,1.5) circle (.05);
    \draw[fill=black] (5.33,0.1) circle (.05);
    \draw[fill=black] (5.33,-1.3) circle (.05);
    \draw (5.33,-1.3) -- (5.33,0.1) -- (5.33,1.5) -- (5.33,2.8) -- (6.4,5.3) node[anchor = south]
    {$(x_0,Z(x_0))$};
    \draw[fill=black] (8.58,2.8+0.8) circle (.05);
    \draw[fill=black] (8.58,1.5+0.8) circle (.05);
    \draw[fill=black] (8.58,0.1+0.8) circle (.05);
    \draw[fill=black] (8.58,-1.3+0.8) circle (.05);
    \draw (7.4,4.5) -- (8.58,2.8+0.8) -- (8.58,1.5+0.8) -- (8.58,0.1+0.8)--(8.58,-1.3+0.8) ;
    \draw (7.4,4.5) node[anchor=south]{$(x_1,Z(x_1))$};
    \draw[fill=black] (9.3,2.8+1.1) circle (.05);
    \draw[fill=black] (9.3,1.5+1.1) circle (.05);
    \draw[fill=black] (9.3,0.1+1.1) circle (.05);
    \draw (9.3,1.2) -- (9.3,2.6) -- (9.3, 3.9) -- (9.0, 5.8) node[anchor=south]{$(x_2,Z(x_2))$};
    \draw[fill=black] (10.03,2.8+1.3) circle (.05);
    \draw[fill=black] (10.03,1.5+1.3) circle (.05);
    \draw (10.03,2.8) -- (10.03, 4.1) -- (10.2, 5.2) node[anchor=south]{$(x_3,Z(x_3))$};
    \draw[fill=black] (11.91,1.5) circle (.05) node[anchor=west]{$(y_0,Z(y_0))$};
    \draw[fill=black] (12.71,3.2) circle (.05) node[anchor=west]{$(y_1,Z(y_1))$};
    \draw[fill=black] (12.61,4.8) circle (.05) node[anchor=west]{$(y_2,Z(y_2))$};
    \draw[fill=black] (11.99,5.2) circle (.05) node[anchor=south west]{$(y_3,Z(y_3))$};
\end{tikzpicture}
\caption{Labelled structure of an $F_4$ approximation (left) with line segments $L_0,L_1,L_2$. Gaussian process indexed by the
$F_4$ subtree (right). From top to bottom, the four graphs show the biased Gaussian process $Z$ along the
geodesic from $x_0$ to $y_3,y_2,y_1$, and $y_0$, respectively.}
    \label{fig:BrownianOnTree}
\end{figure}

  The Gaussian $Z$ on $[0,1]$ has the property that $Z\rvert_{L_i}$ is independent of $Z\rvert_{L_j}$ if
  $i\neq j$ and is a Wiener process on $f(L_i)$.
  Let $s\in L_i, t\in L_j$ and assume without loss of generality that $s<t$. Then, 
  \begin{align*}
    &\Cov(Z(s)-Z(\inf L_i),Z(t)-Z(\inf L_j)) \\
    =& 
    \Cov(Z(s),Z(t))-\Cov(Z(s),Z(\inf L_j))-\Cov(Z(t),Z(\inf L_i))+\Cov(Z(\inf L_i),Z(\inf L_j))\\
    =&\min_{u\in[s,t]}f(u) - \min_{u\in[s,\inf L_j]} f(u) - \min_{[\inf L_i,t]} f(u) + \min_{[\inf
    L_i, \inf L_j]} f(u)\\
    =& \min_{u\in [s,\inf L_j]}f(u) - \min_{u\in [s,\inf L_j]}f(u) - \min_{u\in [\inf L_i, \inf L_j]}f(u)
    +\min_{u\in [\inf L_i, \inf L_j]} f(u)
    =0.
  \end{align*}
  Now let $s,t\in L_i$. Then,
  \begin{equation*}
    \Cov(Z(s), Z(t)) = \min \{f(u)\,:\, \min\{s,t\} \leq u \leq
    \max\{s,t\}\} = f(\min\{s,t\}) =
    \min\{f(s),f(t)\}.
  \end{equation*}
  Let $G_i$ be the unique orientation preserving similarity mapping the interval $[0,1]$ onto $f(L_i)=
  [\max\{f(x_i),f(x_{i+1})\},f(y_i)]$.
  Then, 
  \[Z_i(s) = (Z\circ f^{-1}\circ G_i(s)-Z\circ f^{-1}\circ G_i(0))/\sqrt{\lvert f(L_i)\rvert}\]
  is equal in distribution to a Wiener
  process $W$ and, as $\lvert f(L_i)\rvert \sim \lvert f(L_j)\rvert$, there exists $p_n$ independent of $i$ such that 
  \begin{equation*}
    \Prob\{\lVert Z_i(s) - s\rVert_\infty < 2^{-n}\} \geq p_n>0.
  \end{equation*}
  We can argue similarly for all the joined pieces $J_i$. They are independent (conditioned on
  starting value) and have length at most $(1+2^{1-n}) \lvert f(L_0)\rvert$. Hence the probability that
  the process deviates at most $2^{-n}$ is also comparable to $p_n$ and we assume without loss of
  generality that $p_n$ is a lower bound.
  Thus, the probability that this holds for all $i$ and all connecting tree pieces is at least
  $p_n^{n+1+\log_2 (n+1)}>0$.

  We now establish that these events are independent for distinct intervals. 
  Let $x_0(n,k)$ be as in \eqref{eq:x0def} and let $[s_0(n,k),t_0(n,k)]\subseteq [s_n^k,t_n^k]$ 
  be the smallest interval such that $f(s_0(n,k)) = f(t_0(n,k)) = x_0$. Therefore $f(u)\geq x_0$ for all
  $u\in(s_0(n,k),t_0(n,k))$.
  Assume that at least one of $n\neq n'$ and $k\neq k'$ is given, then the following holds for all
  $u_1, u_2\in
  [s_0(n,k),t_0(n,k)]$ and $v_1,v_2\in [s_0(n',k'),t_0(n',k')]$, assuming without loss of generality
  that $t_0(n,k) \leq s_0(n',k')$,
  \begin{align*}
    &\Cov(Z(u_1)-Z(u_2),Z(v_1)-Z(v_2))\\
    &=\min_{u\in[u_1,v_1]}f(u) - \min_{u\in [u_1,v_2]} f(u)  - \min_{u\in [u_2,v_1]}f(u) + \min_{u\in
    [u_2,v_2]} f(u)\\
    &=\min_{u\in[t_0(n,k),s_0(n',k')]}f(u)
    -\hspace{-1em}\min_{u\in[t_0(n,k),s_0(n',k')]}f(u)-\hspace{-1em}\min_{u\in[t_0(n,k),s_0(n',k')]}f(u)+\hspace{-1em}
    \min_{u\in[t_0(n,k),s_0(n',k')]}f(u)\\
    &=0.
  \end{align*}
  Thus the (relative) Gaussian processes chosen earlier are independent if at least one of $n,n'$ or
  $k,k'$ differ.
  Thus a standard Borel-Cantelli argument shows that there are infinitely many intervals where this
  occurs for every $n$, almost surely.

  The last thing to show is that these structures give rise to $(n-1)$-stars. This involves estimating
  the Brownian map metric as well as the following estimates.
  Let $x,y\in\M$, then
  \[
    Z(x)+Z(y) - 2Z(x\wedge y)\quad \leq\quad D(x,y)\quad \leq\quad D^o(x,y),
    \]
  where $x\wedge y$ is the lowest common ancestor of $x$ and $y$ in $T_{\ex}$
  and $\min_{z\in[x,y]}Z(z) \geq \min_{z\in[0,1]}
  Z(z)$.
  We thus estimate, for $i\neq j$,
  \begin{align*}
    0 \leq D(x_i,x_j) &\leq D^o (x_i,x_j)
    =Z(x_i)+Z(x_j) - 2\min_{s\in[x_i,x_j]} Z(s)\\
    &\leq 2 \max_{k} Z(x_k) - 2 \min_k Z(x_k)
    \leq 4\cdot 2^{-n} \max_k \sqrt{|f(L_k)|}
    \intertext{as well as for the branches}
    D(x_1,y_i) &\leq D(x_1,x_i)+D(x_i,y_i)\\
    &\leq 4\cdot 2^{-n} \max_k\sqrt{|f(L_k)|} + (1+2\cdot 2^{-n})\sqrt{|f(L_i)|}\\
    &\leq (1+6\cdot 2^{-n}) \max_k \sqrt{|f(L_k)|}
    \intertext{and}
    D(y_i,y_j) &\leq D(x_i,x_j)+D(x_i,y_i)+D(x_j,y_j)\\
    &\leq (2+16\cdot 2^{-n}) \max_k \sqrt{|f(L_k)|}.
  \end{align*}
  The lower bounds are given by
  \begin{align*}
    D(x_1,y_i)&\geq Z(y_i)+Z(x_1)-2Z(x_1) = Z(y_i)-Z(x_1)\\
    &\geq Z(y_i)- Z(x_i) - |Z(x_1)-Z(x_i)|\\
    &\geq \sqrt{|f(L_i)|}-2 \cdot 2^{-n} \max_k\sqrt{|f(L_k)|}\\
    &\geq (1-6\cdot 2^{-n}) \max_k\sqrt{|f(L_k)|}
    \intertext{and}
    D(y_i,y_j)&\geq Z(y_i)+Z(y_j)-2Z(y_i\wedge y_j) 
    \geq (2-16\cdot 2^{-n}) \max_k \sqrt{|f(L_k)|}.
  \end{align*}
  Analogous to the proof of Theorem~\ref{thm:CRTStarry} we see that the collection of points $y_i$
  together with centre $x_1$, form an approximate $(n-1)$-star. Since this is true, almost surely,
  for every $n$ we have shown that $(\M,D)$ is starry.
\end{proof}

\section{The continuum tree as a dual space}
\label{sect:dualDiscussion}
While it is not generally true that the operator mapping continuous excursion functions on
$[0,1]$ to continuum trees is invertible, one can see the tree associated with a function as an
indicator of its irregularity. For instance, local maxima of the excursion function are exactly the
leaves of the associated continuum tree (with the exception of the root).
One might conjecture a strong link between dimensions of the graph of an excursion function and its
associated continuum tree. 
\begin{question}\label{thm:question}
  Let $f$ be an excursion function on $[0,1]$ with graph $\mathcal{G}_f=\{(t,f(t)): t\in[0,1]\} \subseteq \R^2$.
  Suppose additionally that $\dim \mathcal{G}_f = s$.
  What can we say about the dimension\footnote{Here, dimension refers to any of the commonly
  considered dimensions: Hausdorff, packing, box-counting, Assouad, and lower dimension.}  $\dim T_f$?
\end{question}
Given the extremal nature of the Assouad dimension one might further conjecture that functions whose
graph has Assouad dimension strictly larger than $1$ have an associated continuum tree that is
starry.
\begin{conjecture}
  Let $f$ be an excursion function such that $\dimA \mathcal{G}_f >1$. Then $T_f$ is starry.
\end{conjecture}

This is certainly not a necessary condition as Example 5.1 below shows.

Establishing such relationships between dimensions and irregular curves is at the heart of fractal
geometry, see e.g.\ \cite{BishopPeres} and \cite{FractalGeo3}.
The dimension theory of the family of Weierstra\ss\
functions has only recently been analysed and self-affine curves are still not fully understood,
see e.g.\ \cite{Barany18}. 
One of the major difficulty for affine functions are complicated relationships between the scaling
in the horizontal as well as vertical axis and the continuum tree might get around this problem by
separating the two scales in terms of placement of nodes and length of subtrees. These may give rise
to a self-similar set in an abstract space other than $\R^d$ that may be easier to understand.
Knowing bounds such as those asked in Question \ref{thm:question} may lead to a better understanding
of the self-affine theory and this ``dual space'' could solve many open
questions for singular functions.

\subsection*{Example 5.1: Excursion with low Assouad dimension whose continuum tree is infinite
dimensional.}\label{ex:lowAss}
Let $s_n(x)$ be the positive triangle function with slope $2$ and $n$ maxima, 
\[
  s_n(x) = \begin{cases} 
    2(x-\tfrac{k}{n})& x\in[\tfrac{k}{n}, \tfrac{k}{n}+\tfrac{1}{2n})\text{ and }0\leq k<n\\
      -2(x-\tfrac{k+1}{n})& x\in[\tfrac{k}{n}+\tfrac{1}{2n}, \tfrac{k+1}{n}) \text{ and }0\leq k <
	n\\
	0& x=1
      \end{cases}
\]
Let us define disjoint dyadic intervals $[a_n,b_n]\subset[0,1/2]$ by $a_n = 1/2-2^{-(n+1)}$ and $b_n
= a_n + 2^{-(n+3)}$.
 Note that $a_{n+1}-b_n = b_n-a_n$, and so, the gap between $[a_n,b_n]$
and $[a_{n+1},b_{n+1}]$ is equal in size to the length of the former interval. We define
\[
  s(x) = \sum \chi_{[a_n,b_n]}(x)\cdot (b_n-a_n)\cdot s_{(n+1)^n}((x-a_n)/(b_n-a_n))
\]
and note that this function is triangular on $[a_n,b_n]$ with slope $2$ but decreasing amplitude and
increasing frequency, such that it has $(n+1)^n$ maxima on $[a_n,b_n]$. The map $S:(x,y)\mapsto
(x,y+s(x))$ is easily shown to be bi-Lipschitz on $[0,1]\times \R$ with Lipschitz constant
$\sqrt{5}$. Further, let
\[
  g_1(x) = \int_0^x \left(1-\sum_{n=1}^\infty \chi_{[a_n,b_n]}(y)\right)dy.
\]
It is clear that $g_1(x)$ is a non-decreasing function with slope $0$ in $[a_n,b_n]$ and slope $1$
otherwise. Let $g_2(x) = -2g_1(1/2) x + 2g_1(1/2)$ and
\[
  g(x) = \begin{cases} g_1(x) & 0\leq x < 1/2\\
    g_2(x) & 1/2\leq x \leq 1
  \end{cases}
\]
Similarly, $G:(x,y)\mapsto (x,y+g(x))$ is bi-Lipschitz on $[0,1]\times\R$ with Lipschitz constant
$\sqrt{2}$.
It can be checked that $g(x)+s(x)$ is an excursion function on $[0,1]$ see also
Figure~\ref{fig:allexcursion}.
The graph $\{(x,g(x)+s(x)):x\in[0,1]\}$ is the image of $[0,1]\times\{0\}$ under $S\circ G$ and hence is a bi-Lipschitz image
of the unit line. Therefore the Assouad dimension of the excursion graph $g(x)+s(x)$ is $1$, i.e.\
as low as possible for a continuous function.
To see that the resulting continuum tree $T_{g(x)+s(x)}$ is starry, one needs to observe that the
triangle function gives the excursion functions $(n+1)^n$ peaks in $[a_n,b_n]$ whose amplitude is
small compared to $g(a_{n+1})$. This gives rise to approximate $(n+1)^n$-stars centred at $a_n\in
T_{g(x)+s(x)}$ for all $n$ and thus $T_{g(x)+s(x)}$ is starry. See Figure~\ref{fig:exampleset} for
an illustration of $T_{g(x)+s(x)}$.

\noindent
\begin{minipage}[c]{.45\textwidth}
  \begin{minipage}[c]{1.0\textwidth}
        \includegraphics[width=\textwidth]{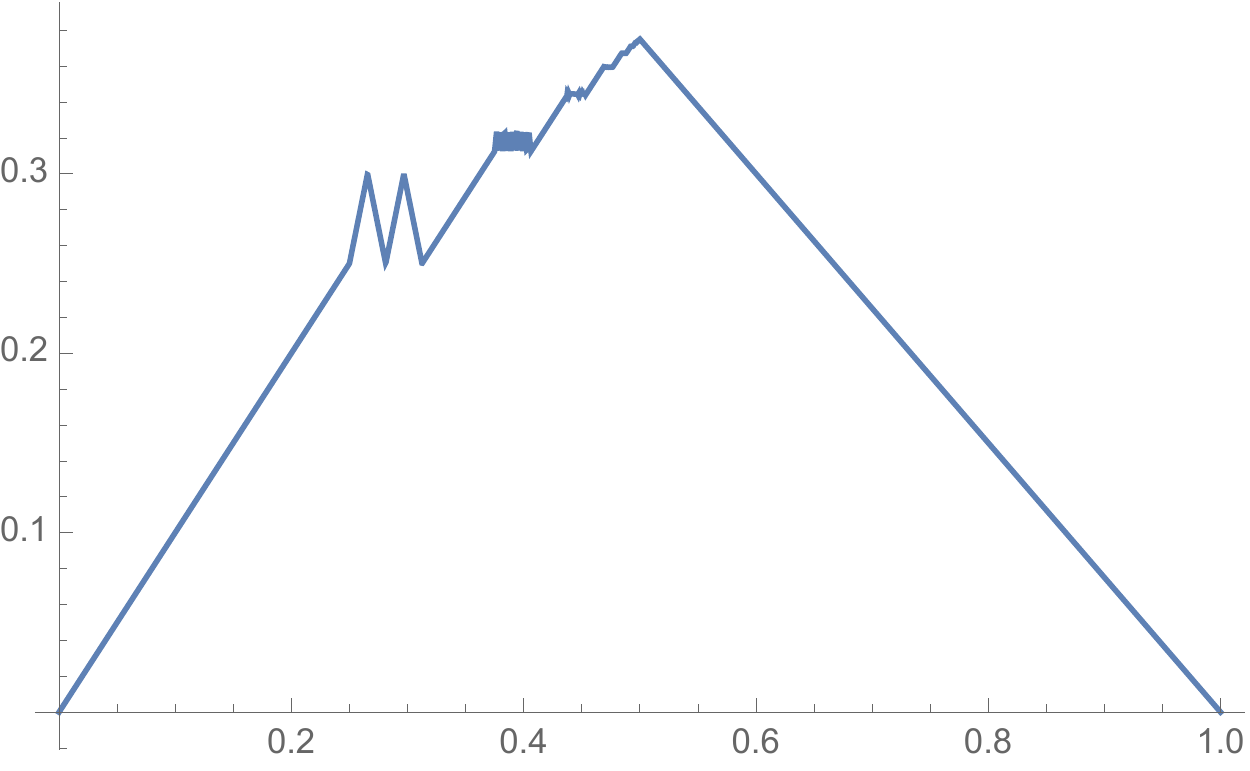}
	\captionof{figure}{The excursion function $g(x)+s(x)$.}
        \label{fig:allexcursion}
      \end{minipage}

\begin{minipage}[c]{1.0\textwidth}
        \includegraphics[width=\textwidth]{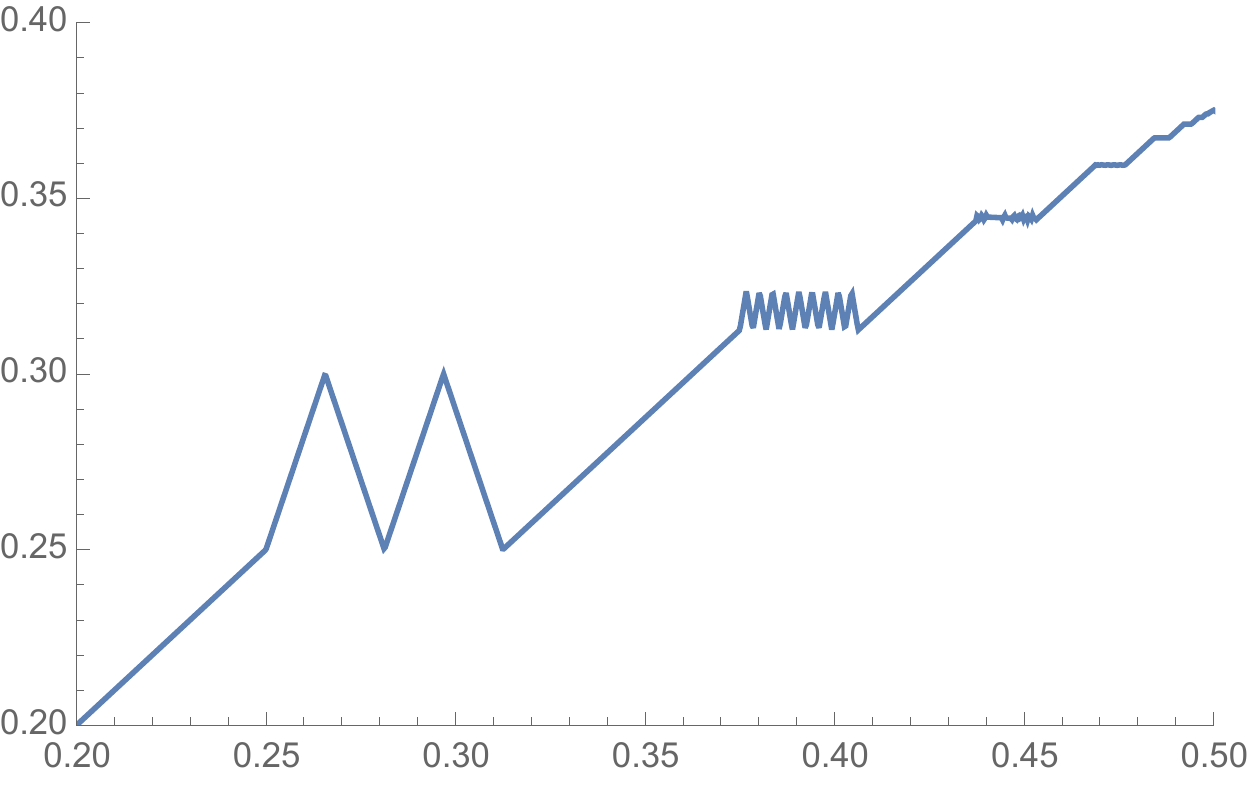}
	\captionof{figure}{Close up of Figure~\ref{fig:allexcursion}.}
        \label{fig:zoomin}
      \end{minipage}
\end{minipage}
\begin{minipage}[c]{.56\textwidth}
        \includegraphics[width=\textwidth]{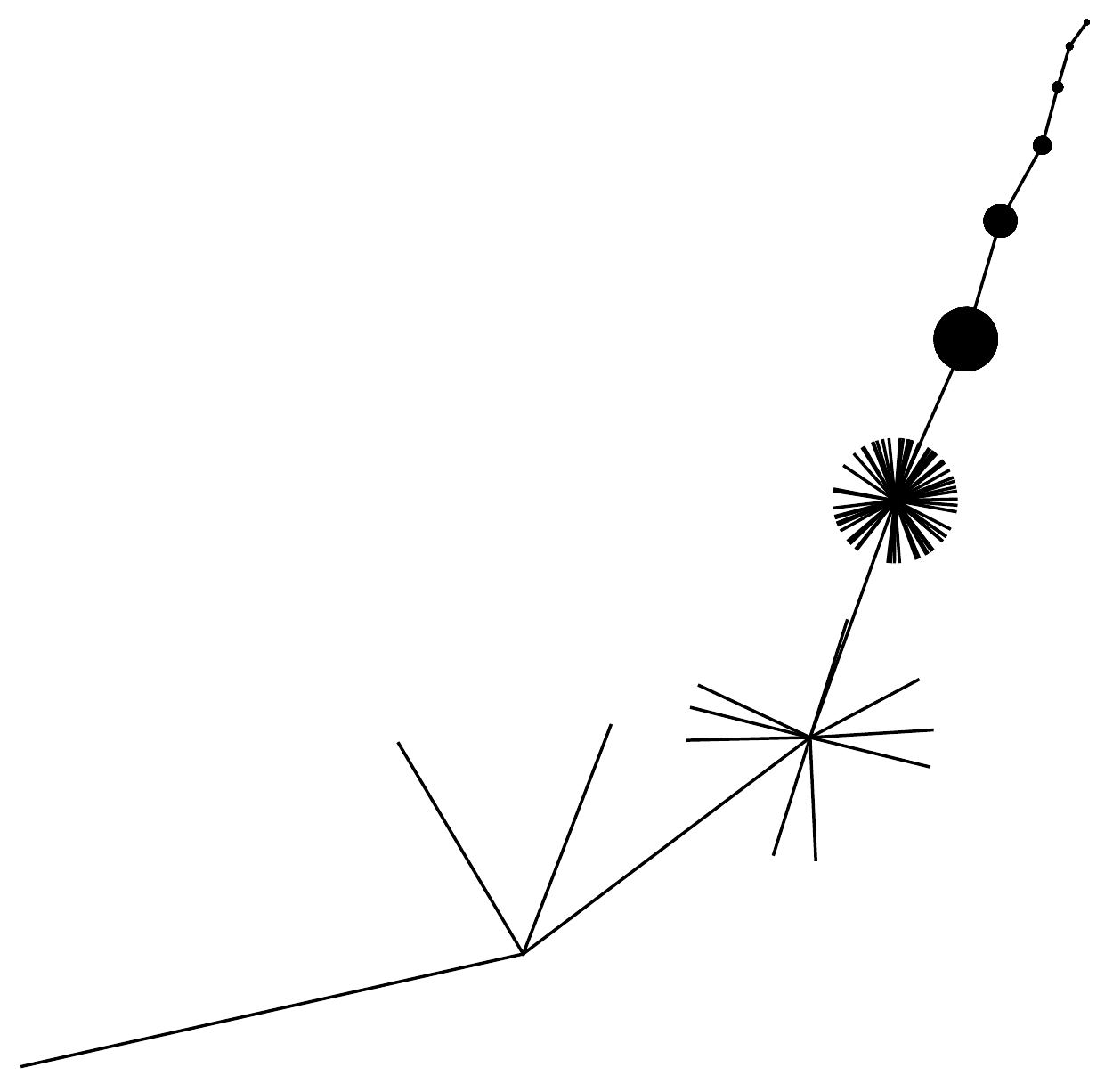}
	\captionof{figure}{The continuum tree $T_{g(x)+s(x)}$.}
        \label{fig:exampleset}
 \end{minipage}

\subsection*{Acknowledgements}
The work on the BCRT first arose from a conversation with Noah Forman while the author was visiting
the University of Washington in April 2018. 
ST thanks the University of Washington, and Jayadev Athreya and Noah Forman in particular, for the
financial support, hospitality, and inspiring research atmosphere.
The extension to the Brownian map was inspired by a
question of Xiong Jin at the \emph{Thermodynamic Formalism, Ergodic Theory and Geometry Workshop} at
Warwick University in July 2019 and ST thanks Xiong Jin for many helpful conversations.
Finally, I thank the referee for their helpful comments on an earlier version of this manuscript.

\end{document}